\documentclass[reqno,11pt]{amsart}

\usepackage{amsmath,amssymb,latexsym,soul,cite,mathrsfs,amsfonts}
\usepackage{color,enumitem,graphicx}
\usepackage[colorlinks=true,urlcolor=blue,
citecolor=red,linkcolor=blue,linktocpage,pdfpagelabels,
bookmarksnumbered,bookmarksopen]{hyperref}
\usepackage[english]{babel}
\usepackage{comment}
\usepackage{tikz}
\usetikzlibrary{intersections}
\usepackage[symbol]{footmisc}

\usepackage[left=2.9cm,right=2.9cm,top=2.8cm,bottom=2.8cm]{geometry}
\usepackage[hyperpageref]{backref}

\usepackage[colorinlistoftodos]{todonotes}
\makeatletter
\providecommand\@dotsep{5}
\def\listtodoname{List of Todos}
\def\listoftodos{\@starttoc{tdo}\listtodoname}
\makeatother

\usepackage{verbatim} % to using the command "\comment"

\numberwithin{equation}{section}
\newtheorem{theorem}{Theorem}[section]
\newtheorem{proposition}[theorem]{Proposition}
\newtheorem{lemma}[theorem]{Lemma}
\newtheorem{corollary}[theorem]{Corollary}
\newtheorem{remark}[theorem]{Remark}

\newcommand\restr[2]{{% we make the whole thing an ordinary symbol
  \left.\kern-\nulldelimiterspace % automatically resize the bar with \right
  #1 % the function
  \vphantom{\big|} % pretend it's a little taller at normal size
  \right|_{#2} % this is the delimiter
  }}
\title[Zero energy critical points]
{Zero energy critical points of functionals depending on a parameter}

\author[H. Ramos Quoirin]{Humberto Ramos Quoirin}

\author[J. Silva]{Jefferson Silva}

\author[K. Silva]{Kaye Silva*} \thanks{*Corresponding author.}

\address{H. Ramos Quoirin  \newline \indent CIEM-FaMAF \newline \indent Universidad Nacional de C\'{o}rdoba, (5000)
	C\'{o}rdoba, Argentina}
\email{\tt humbertorq@gmail.com}

\address{J. Silva \newline\indent
	Instituto de Matem\'atica e Estat\'istica.   
	\newline\indent 
	Universidade Federal de Goi\'as,
	\newline\indent
	Rua Samambaia, 74001-970, Goi\^ania, GO, Brazil}
\email{\href{mailto:	jeffersonsantos.mat@gmail.com}{	jeffersonsantos.mat@gmail.com}}

\address{K. Silva  \newline\indent
	Instituto de Matem\'atica e Estat\'istica.   
	\newline\indent 
	Universidade Federal de Goi\'as,
	\newline\indent
Rua Samambaia, 74001-970, Goi\^ania, GO, Brazil}
\email{\href{mailto:kayesilva@ufg.br}{kayesilva@ufg.br}}

\thanks{Kaye Silva was partially supported by CNPq/Brazil under Grant 408604/2018-2.}


\begin{document}

\begin{abstract}
We investigate zero energy critical points for a class of functionals $\Phi_\mu$ defined on a uniformly convex Banach space, and depending on a real parameter $\mu$. More precisely, we show the existence of a sequence $(\mu_n)$ such that $\Phi_{\mu_n}$ has a pair of critical points $\pm u_n$ satisfying $\Phi_{\mu_n}(\pm u_n)=0$, for every $n$. In addition, we provide some properties of $\mu_n$ and $u_n$. This result, which is proved via a fibering map approach (based on the {\it nonlinear generalized Rayleigh quotient} method \cite{I1}) combined with the Ljusternik-Schnirelman theory, is then applied to several classes of elliptic pdes.
\end{abstract}

\bigskip

\maketitle
\begin{center}
\begin{minipage}{12cm}
\tableofcontents
\end{minipage}
\end{center}

\bigskip

\maketitle
\section{Introduction}
\strut

This article is devoted to the existence of critical points with zero energy for a class of functionals depending on a real parameter.  Let $X$ be a uniformly convex Banach space equipped with $\|\cdot \| \in C^1(X \setminus \{0\})$, and $\Phi_\mu:X\to \mathbb{R}$  be the energy functional given by
\begin{equation*}
	\Phi_\mu=I_1-\mu I_2,
\end{equation*}
where $\mu\in \mathbb{R}$ and $I_1,I_2$ are even $C^1$ functionals satisfying $I_1(0)=I_2(0)=0$. Assume in addition that $I_2(u)\neq 0$ for each $u\in X\setminus\{0\}$, so that for any such $u$ we have $\Phi_\mu(u)=0$ if and only if 
\begin{equation*}
	\mu=\mu_0(u):=\frac{I_1(u)}{I_2(u)}.
\end{equation*}
This functional is obtained by the {\it nonlinear generalized Rayleigh quotient} method, which has been introduced by Il'yasov \cite{I1,I3}. The interest in the functional $\mu_0$ lies in the following relation, cf. \cite{I3}:
$$\mu_0'(u)=\frac{\Phi_{\mu_0(u)}'(u)}{I_2(u)}, \quad \forall u\neq 0.$$
Thus the couple $(\mu,u)$ satisfies $\Phi_\mu(u)=\Phi'_\mu(u)=0$ if, and only if, $\mu_0'(u)=0$ and $\mu_0(u)=\mu$.
We shall look for critical points (and critical values) of $\mu_0$ by considering the fibering maps associated to $\mu_0$, i.e. $\psi_{u}(t):=\mu_0(tu)$, defined for any $u\in X\setminus\{0\}$ and $t>0$. Let us assume the following condition:\\
\begin{itemize}
	\item[(H1)] $\psi_u \in C^2(0,\infty)$ for every $u \in X \setminus\{0\}$, $u \mapsto \psi'_u(t) \in C^1(X \setminus\{0\})$ for every $t>0$, and $\psi_{u}$ has at most one critical point $t_0(u)>0$, which is a non-degenerate global maximum point. Moreover, $D=\{u\in X\setminus\{0\}: t_0(u)\ \mbox{is defined}\}$ is a non-empty open set. \\
\end{itemize}

Our main purpose is to obtain couples $(\mu,u)$ solving the system $\Phi_\mu(u)=\Phi'_\mu(u)=0$  by looking for critical points of the even functional $\Lambda: D \to \mathbb{R}$ given by
\begin{equation*}
	\Lambda(u):=\psi_u(t_0(u))=\mu_0(t_0(u)u).
\end{equation*}
As a matter of fact, we shall see that critical points of $\Lambda$ (along with its critical values) provide all solutions of the system above. In other words,  $u$ is a critical point of  $\Lambda$ if, and only if, $t_0(u)u$ is a zero energy critical point of $\Phi_{\mu}$ with $\mu=\Lambda(u)$.  Such critical points will be found by dealing with $\tilde{\Lambda}$, the restriction of $\Lambda$ to $S \cap D$, where $S$ is the unit sphere in $X$.  It follows from (H1) that $S\cap D$ is a symmetric $C^1$ manifold (the most simple case being $D=X \setminus \{0\}$, so that $S \cap D=S$). It turns out that $S \cap D$ is given by a {\it natural constraint}, so that critical points of $\tilde{\Lambda}$ are also critical points of $\Lambda$. 
%Since one may also deal with $(\tilde{\Lambda})^{-1}=1/\tilde{\Lambda}$ to find critical points of $\Lambda$,
 We shall assume the following condition:\\
\begin{itemize}		
	\item[(H2)] $\Lambda$ is bounded, either from below or from above.\\
\end{itemize}

Let us recall that $\tilde{\Lambda}$ satisfies the Palais-Smale condition if any sequence $(u_n) \subset S \cap D$ such that $(\tilde{\Lambda}(u_n))$ is bounded and $\tilde{\Lambda}'(u_n) \to 0$ has a subsequence converging to some $u \in S \cap D$.

Critical points of  $\tilde{\Lambda}$ shall be obtained via the Ljusternik-Schnirelman principle. Given a nonempty symmetric and closed set $F \subset S \cap D$, we recall that $$\gamma(F):=\inf\{n \in \mathbb{N}: \exists h:F\to  \mathbb{R}^n \setminus \{0\} \mbox{ odd and continuous} \}$$ is the Krasnoselskii genus of $F$.  For every $n \in \mathbb{N}$ we set
\begin{equation}\label{dfn} \mathcal{F}_n=\{F\subset S\cap D: F \mbox{ is compact, symmetric, and } \gamma(F)\ge n\},\end{equation} and
\begin{equation*}
	\mu_n:=\begin{cases}\displaystyle
	\inf_{F\in \mathcal{F}_n}\sup_{u\in F}\Lambda(u), & \mbox{ if  $\Lambda$ is bounded from below}\\
	\displaystyle \sup_{F\in \mathcal{F}_n}\inf_{u\in F}\Lambda(u), & \mbox{ if $\Lambda$ is bounded from above.}
	 \end{cases}
\end{equation*}

%\begin{theorem}\label{THM1} Suppose that (H2),(H2) hold, and  $\gamma(D\cap S)=k\in\mathbb{N}$.
%	\begin{itemize}
%		\item[I)] If $(H2)$ holds then there exists $\mu_1<\mu_2<\cdots<\mu_k$ and $k$ pairs $\pm u_1,\pm u_2,\cdots,\pm u_k\in X\setminus\{0\}$ such that $\Phi_{\mu_n}(\pm u_n)=0$ and $\Phi'_{\mu_n}(\pm u_n)=0$ for $n=1,\cdots,k$.
%			\item[II)] If $(H2)$ holds then there exists $\mu_k<\mu_{k-1}<\cdots<\mu_1$ and $k$ pairs $\pm u_1,\pm u_2,\cdots,\pm u_k\in X\setminus\{0\}$ such that $\Phi_{\mu_n}(\pm u_n)=0$ and $\Phi'_{\mu_n}(\pm u_n)=0$ for $n=1,\cdots,k$.
%	\end{itemize} 
%\end{theorem}
%In the next Theorem, for simplicity of notation, we denote a subsequence of $(\mu_n)$ by the same symbol.
 The previous discussion leads to our main result:
\begin{theorem}\label{THM1} Suppose that (H1) and (H2) hold.
\begin{enumerate}
\item If $\Lambda$ is bounded from below (respect. above) and $\mu<\mu_1=\displaystyle \inf_D \Lambda$ (respect. $\mu>\mu_1=\displaystyle \sup_D \Lambda$) then there exists no $u \in X \setminus \{0\}$ such that  $\Phi_\mu(u)=\Phi'_\mu(u)=0$.

\item Suppose that $\gamma(D\cap S)=\infty$ and $\tilde{\Lambda}$ satisfies the Palais-Smale condition at $\mu_n$ for every $n \in \mathbb{N}$. Then there exists a sequence of pairs $(\pm u_n)\subset X\setminus\{0\}$ such that $\Phi_{\mu_n}(\pm u_n)=0$ and $\Phi'_{\mu_n}(\pm u_n)=0$ for every $n$. Moreover $(\mu_n)$ is an increasing  (respect. decreasing) sequence if $\Lambda$ is bounded from below (respect. above).
\end{enumerate}
\end{theorem}

This result applies to several elliptic problems, as we shall see in Section 4. Among these problems, the most classical one is the concave-convex equation 
\begin{equation}
	\label{CC}
\left\{
\begin{array}
	[c]{lll}%
	-\Delta u =\mu |u|^{q-2}u+|u|^{r-2}u & \mathrm{in} & \Omega,\\
u=0  & \mathrm{on} & \partial\Omega,
\end{array}
\right. 
\end{equation}
where $\Omega\subset \mathbb{R}^N$ is a bounded domain, $\mu>0$, and $1<q<2<r<2^*$. We shall see that the energy functional of \eqref{CC} satisfies (H1) and (H2), with $\Lambda$ bounded from below. 

Theorem \ref{THM1} yields a sequence $(\mu_n,u_n) \subset (0,\infty) \times H_0^1(\Omega) \setminus \{0\}$ such that $u_n$ solves \eqref{CC} for $\mu=\mu_n$ and $\Phi_{\mu_n}(u_n)=0$, for every $n$. Moreover $\mu_n \to \infty$ as $n \to \infty$. On the other hand, it is known \cite{ABC,BW} that for any $\overline{\mu}>0$  this problem has infinitely many solutions $w_n$ whose energy grows up to infinity as $n \to \infty$, as well as infinitely many solutions $\tilde{w}_n$ whose energy converges to zero as $n \to \infty$. Moreover, from the definition of $\mu_1$ and \cite[Theorems 1.3 and 1.4]{I2}, one can find two decreasing and continuous branches of positive solutions, i.e. solutions $u_\mu,v_\mu$ such that $\mu \mapsto \Phi_{\mu}(u_\mu),\Phi_{\mu}(v_\mu)$ are continuous and decreasing in $(0,\mu^*)$, for some $\mu_1<\mu^*<\infty$ (see the full blue and red curves in Figure \ref{fig:M1}, respectively). In addition, $\Phi_{\mu}(u_\mu)>0$ if $\mu\in(0,\mu_1)$, $\Phi_{\mu_1}(u_{\mu_1})=0$ and $\Phi_{\mu}(u_\mu)<0$ if $\mu\in(\mu_1,\mu^*)$, whereas $\Phi_{\mu}(v_\mu)<\min\{0,\Phi_{\mu}(u_\mu)\}$ if $\mu\in (0,\mu^*)$.   

From the previous results one may conjecture that for every $n$ there exists a decreasing and continuous solution branch (the dashed blue curves in Figure \ref{fig:M1}) that reaches the zero level at $\mu=\mu_n$ and a corresponding decreasing and continuous solution branch (the dashed red curves in Figure \ref{fig:M1}) that exists at least for $0\leq \mu<\mu_n$. So the intersection points of the line $\mu=\overline{\mu}$ with these branches would correspond to the solutions $w_n,\tilde{w}_n$ found in \cite{ABC,BW}.  We also note that if $\Omega$ is an annulus then there exists infinitely many global continua $\mathcal{C}_n$ of solutions of \eqref{CC} bifurcating from the point $(\mu,u)=(0,0)$, see \cite[Figure 1 and Theorem 2.1]{BM}.  Moreover, the energy of the solutions $\underline{u}_n$ (respect. $\overline{u}_n$) in the lower (respect. upper) part of $\mathcal{C}_n$ tends to 0 (respect. $\infty$) as $n \to \infty$. We believe that in this case the solutions $\underline{u}_n$ (respect. $\overline{u}_n$) would correspond to the solutions in the red (respect. blue) curves of Figure 1 below. Furthermore, each red curve would meet the corresponding blue curve at a common endpoint, i.e. these two curves would be a continuum of solutions.

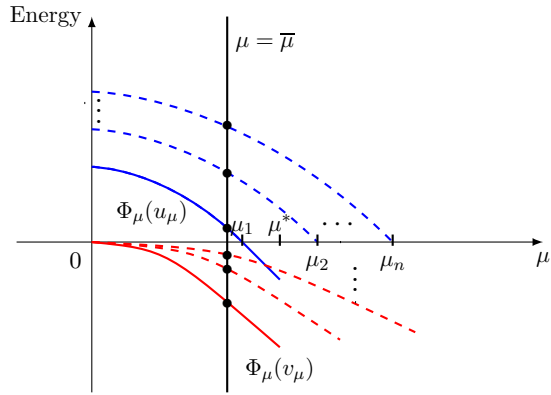
\begin{figure}[h]
	\centering
	\begin{tikzpicture}[>=latex]
		%x axis
		\draw[->] (-1,0) -- (6,0) node[below] {\scalebox{0.8}{$\mu$}};
		\foreach \x in {}
		\draw[shift={(\x,0)}] (0pt,2pt) -- (0pt,-2pt) node[below] {\footnotesize $\x$};
		%y axis
		\draw[->] (0,-2) -- (0,3) node[left] {\scalebox{0.8}{$\mbox{Energy}$}};
		\foreach \y in {}
		\draw[shift={(0,\y)}] (2pt,0pt) -- (-2pt,0pt) node[left] {\footnotesize $\y$};
		\node[below left] at (0,0) {\footnotesize $0$};
		%	\draw[red,thick] (0,-2) .. controls (0,-1.5) and (0,0) .. (4,1.5);
		\draw[blue,thick] (0,1) .. controls (0,1) and (1,1) .. (2,0);
		\draw[blue,thick] (2,0) .. controls (2.5,-.5) .. (2.5,-0.5);
		\draw[blue,thick,dashed] (0,1) .. controls (0,1) and (1,1) .. (2,0);
		\draw[blue,thick,dashed] (0,1.5) .. controls (0,1.5) and (1.5,1.5) .. (3,0);
		\draw [thick] (3.3,0) node[above]{$\cdots$} -- (3.3,0); 
		\draw[blue,thick,dashed] (0,2) .. controls (0,2) and (2,2) .. (4,0);
		\draw[red,thick] (0,0) .. controls (1,-.1) .. (2.5,-1.4);
		\draw[red,thick,dashed] (0,0) .. controls (1.5,-.1) .. (3.3,-1.3);
		\draw[red,thick,dashed] (0,0) .. controls (2,-0.1) .. (4.3,-1.2);
		\draw [thick] (3.5,-.8) node[above]{$\vdots$} -- (3.5,-.8); 
		\draw [thick] (-.1,1.85) node[right]{$\vdots$} -- (-.1,1.85);
		\draw [thick] (1.8,-2) -- (1.8,3);
		%\draw (.5,.35) node{$\bullet$};
		\draw (1.8,.18) node{\scalebox{0.8}{$\bullet$}};
		\draw (1.8,.91) node{\scalebox{0.8}{$\bullet$}};
		\draw (1.8,1.54) node{\scalebox{0.8}{$\bullet$}};
			\draw (1.8,-.18) node{\scalebox{0.8}{$\bullet$}};
			\draw (1.8,-.37) node{\scalebox{0.8}{$\bullet$}};
			\draw (1.8,-.82) node{\scalebox{0.8}{$\bullet$}};
		%	\draw[blue,thick,dashed] (3,1.3) .. controls (3,1.3) and (4,1.5) .. (4,1.5);
		%\draw [thick] (1,-.1) node[below]{$\mu_1$} -- (1,0.1); 
		\draw [thick] (2,-.05) node[above]{\scalebox{0.8}{$\mu_1$}} -- (2,0.1); 
		\draw [thick] (3,-.05) node[below]{\scalebox{0.8}{$\mu_2$}} -- (3,0.1); 
		\draw [thick] (4,-.05) node[below]{\scalebox{0.8}{$\mu_n$}} -- (4,.1); 
		%\draw [thick] (1.8,0) node[above left]{\scalebox{0.8}{\tiny$\overline{\mu}$}} -- (1.8,0); 
			\draw [thick] (2.5,-0.05) node[above]{\scalebox{0.8}{$\mu^*$}} -- (2.5,0.1); 
		%	\draw [thick] (-.1,-2) node[left]{$-\infty$} -- (.1,-2); 
		%	\node[] at (1,1.5) { {\color{blue}$\Phi_\mu(w_\mu)$}};
		%	\node[] at (3,0.5) { {\color{red}$\Phi_\mu(u_\mu)$}};
		\draw  (.8,0.1) node[above]{\scalebox{0.8}{$\Phi_\mu(u_\mu)$}} ; 
	\draw  (2.5,-2) node[above]{\scalebox{0.8}{$\Phi_\mu(v_\mu)$}} ; 
	\draw  (2.3,2.4) node[above]{\scalebox{0.8}{$\mu=\overline{\mu}$}} ;
	\end{tikzpicture}
	\caption{Hypothetical energy curves depending on $\mu$ for \eqref{CC}} \label{fig:M1}
	
\end{figure}

%A large number of results on this problem has been established over the past decades, most of them for a fixed $\mu>0$. In particular, many of these results provide the existence of solutions as well as informations on their energy, which is given by the functional 
%\begin{equation*}
%	\Phi_\mu(u)=\frac{1}{p}\int_\Omega |\nabla u|^2-\frac{\mu}{q}\int_\Omega |u|^q-\frac{1}{r}\int_\Omega |u|^r, \quad u \in H_0^1(\Omega).
%\end{equation*}  
%We say that $\Phi_\mu(u)$ is the energy of a solution $u$ of \eqref{CC}. 
%In \cite{ABC} the authors proved that for any $\mu>0$ sufficiently small there are a sequence of solutions with negative energy, and a sequence of solutions with positive energy. This result was extended to any $\mu>0$ in \cite{BW}. Moreover, it is proved \cite{BW} that the energy of these sequences approaches $0$ and $\infty$, respectively.
%Here we shall prove the existence of a sequence $\mu_n \subset (0,\infty)$ such that \eqref{CC} has a pair of solutions $\pm u_n$ with $I_{\mu_n}(\pm u_n)=0$ for every $n$.

A similar (and dual) scenario seems to arise in the Schrödinger-Poisson type problem
\begin{equation}\label{SP}
	\left\{
	\begin{array}
		[c]{lll}%
		-\Delta u+\omega u+\mu \phi u=|u|^{p-2}u & \mathrm{in} & \mathbb{R}^3,\\
		-\Delta\phi+a^2\Delta ^2u=4\pi u^2  & \mathrm{in} & \mathbb{R}^3,
	\end{array}
	\right. 
\end{equation}
where $p\in(2,3)$, $\omega>0$, and $a\ge 0$. This is a prototype of a second class of problems, where (H2) is now satisfied by $(\tilde{\Lambda})^{-1}$.

 When $a=0$ this problem reduces to the so called Schrödinger-Poisson system in $\mathbb{R}^3$. The reduction method, introduced in \cite{BF}, allows one to treat  this system as a single equation by noting that if
\begin{equation*}
	\phi_{0,u}=\frac{1}{|\cdot |} * u^2,
\end{equation*}
then the system is equivalent to
\begin{equation}\label{SPRedu}
	-\Delta u+\omega u+\mu \phi_{0,u}u=|u|^{p-2}u.
\end{equation}  A detailed study of equation \eqref{SPRedu} in the regime $p\in (2,6)$ has been carried out in  \cite{R}. The author showed that when $p\in (2,3)$ there exist two  positive solutions $u_\mu,w_\mu$ satisfying $\Phi_{\mu}(u_\mu)<0<\Phi_{\mu}(w_\mu)$ for $\mu>0$ small enough. These results were extended and generalized in \cite{DS,SS} to the problem \eqref{SP} with $a>0$, i.e., the Schrödinger equations coupled with the Bopp-Podolski Lagrangian of the electromagnetic field. A similar reduction method holds here: if
\begin{equation*}
	\phi_{a,u}=\frac{1-e^{-|\cdot|/a}}{|\cdot |}* u^2,
\end{equation*}
then the system is equivalent to 
\begin{equation}\label{SPReduGK}
	-\Delta u+\omega u+\mu \phi_{a,u}u=|u|^{p-2}u.
\end{equation}
It was shown in \cite{SS} that there exists $\mu_0^*>0$ and two increasing branches of positive solutions $u_\mu,v_\mu$ such that $\Phi_{\mu}(u_\mu)<0$ if $\mu\in(0,\mu_0^*)$, $\Phi_{\mu_0^*}(u_{\mu_0^*})=0$ and $\Phi_{\mu}(u_\mu)>0$ if $\mu\in(\mu_0^*,\mu_0^*+\varepsilon)$ for some $\varepsilon>0$, whereas $\Phi_\mu(v_\mu)>\max\{0,\Phi_\mu(u_\mu)\}$ for $\mu\in(0,\mu_0^*+\varepsilon)$, see the red and blue full curves  in Figure \ref{fig:M2}. By Theorem \ref{THM1} we shall obtain a sequence $(\mu_n,u_n) \subset (0,\infty) \times H_r^1(\mathbb{R}^3)\setminus \{0\}$ such that $u_n$ solves \eqref{SP} for $\mu=\mu_n$ and $\Phi_{\mu_n}(u_n)=0$, for every $n$. Moreover $\mu_n \to 0$ as $n \to \infty$. Note also that $\Phi_\mu$  has no nontrivial critical points for $\mu$ large enough, cf. \cite[Theorem 1.1]{SS}. As we will see, we have $\mu_1=\mu_0^*$, for a suitable choice of $I_1,I_2$.

  Thus the energy picture seems now to be like in Figure \ref{fig:M2}: for each $n$ there would be a red dashed increasing curve crossing the $\mu$ axis at $\mu=\mu_n$, and a blue dashed increasing curve, which would exist at least for $0\leq \mu<\mu_n$. One may also conjecture that for each $k\in \mathbb{N}$ and $\mu$ sufficiently small, there exist at least $2k$ solutions, which correspond to the black dots in the figure. Let us add that a similar structure seems to arise in a Kirchhoff equation (see subsection \ref{sk} below). 

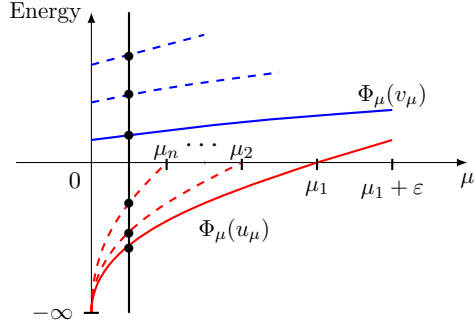
\begin{figure}[h]
	\centering
	\begin{tikzpicture}[>=latex]
		%x axis
		\draw[->] (-1,0) -- (5,0) node[below] {\scalebox{0.8}{$\mu$}};
		\foreach \x in {}
		\draw[shift={(\x,0)}] (0pt,2pt) -- (0pt,-2pt) node[below] {\footnotesize $\x$};
		%y axis
		\draw[->] (0,-2) -- (0,2) node[left] {\scalebox{0.8}{$\mbox{Energy}$}};
		\foreach \y in {}
		\draw[shift={(0,\y)}] (2pt,0pt) -- (-2pt,0pt) node[left] {\footnotesize $\y$};
		\node[below left] at (0,0) {\footnotesize $0$};
		\draw[red,thick] (0,-2) .. controls (0,-1) and (1.5,-.5) .. (3,0);
		\draw[red,thick] (3,0) -- (4,.3);
		\draw[red,thick,dashed] (0,-2) .. controls (0,-1) and (1,-.5) .. (2,0);
		\draw[red,thick,dashed] (0,-2) .. controls (0,-1) and (.5,-.5) .. (1,0);
		\draw [thick] (.5,2) -- (.5,-2);
		\draw[blue,thick] (0,.3) .. controls (2,0.55) .. (4,.7);
			\draw[blue,thick,dashed] (0,.8) .. controls (1,1) .. (2.5,1.2);
				\draw[blue,thick,dashed] (0,1.3) .. controls (.75,1.5) .. (1.5,1.7);
	%	\draw[blue,thick,dashed] (3,1.3) .. controls (3,1.3) and (4,1.5) .. (4,1.5);
	\draw [thick] (1,-.1) node[above]{\scalebox{0.8}{$\mu_n$}} -- (1,0.05); 
		\draw [thick] (2,-.1) node[above]{\scalebox{0.8}{$\mu_2$}} -- (2,0.05); 
		\draw [thick] (3,-.1) node[below]{\scalebox{0.8}{$\mu_1$}} -- (3,0.05); 
		\draw [thick] (4,-.1) node[below]{\scalebox{0.8}{$\mu_1+\varepsilon$}} -- (4,0.05); 
		\draw [thick] (-.1,-2) node[left]{\scalebox{0.8}{$-\infty$}} -- (.1,-2); 
		\draw [thick] (1.5,0) node[above]{$\cdots$} -- (1.5,0); 
		\draw (.5,-.55) node{\scalebox{0.8}{$\bullet$}};
		\draw (.5,-.95) node{\scalebox{0.8}{$\bullet$}};
		\draw (.5,-1.15) node{\scalebox{0.8}{$\bullet$}};
		\draw (.5,.35) node{\scalebox{0.8}{$\bullet$}};
		\draw (.5,.9) node{\scalebox{0.8}{$\bullet$}};
		\draw (.5,1.4) node{\scalebox{0.8}{$\bullet$}};
	\draw  (1.9,-1.2) node[above]{\scalebox{0.8}{$\Phi_\mu(u_\mu)$}} ; 
	\draw  (4,.6) node[above]{\scalebox{0.8}{$\Phi_\mu(v_\mu)$}} ; 
		%\node[] at (1,1.5) { {\color{blue}$\Phi_\mu(w_\mu)$}};
	%	\node[] at (3,0.5) { {\color{red}$\Phi_\mu(u_\mu)$}};
	\end{tikzpicture}
	\caption{Hypothetical energy curves depending on $\mu$ for \eqref{SP}} \label{fig:M2}
	
\end{figure}

\begin{remark}
Let us consider the Nehari set associated to $\Phi_\mu$, i.e. $\mathcal{N}_{\mu}:=\{u \in X \setminus \{0\}: J_\mu(u)=0\}$, and the manifolds  $\mathcal{N}^{\pm}_\mu:=\{u \in \mathcal{N}_\mu:  J_{\mu}'(u)u \gtrless 0\}$. Here $J_\mu(u):=\Phi_\mu'(u)u$ for $u \in X$.
The solutions $u_{\mu}$ and $v_{\mu}$ of \eqref{CC} are obtained as least energy solutions over $\mathcal{N}_\mu^{-}$ and $\mathcal{N}_\mu^{+}$, respectively. In a similar way, the solutions $u_{\mu}$ and $v_{\mu}$ of \eqref{SP} are  least energy solutions over $\mathcal{N}_\mu^{+}$ and $\mathcal{N}_\mu^{-}$, respectively. It is then expected that the blue curves in Figures \ref{fig:M1} and \ref{fig:M2} lie in $\mathcal{N}_\mu^{-}$, whereas the red ones lie in $\mathcal{N}_\mu^{+}$. It turns out that the zero energy solutions in Theorem \ref{THM1} belong indeed to $\mathcal{N}_{\mu_n}^{-}$ for \eqref{CC}, and to $\mathcal{N}_{\mu_n}^{+}$ for \eqref{SP}, which is consistent with this analysis. 
\end{remark}

 To the best of our knowledge, a systematic study of a functional satisfying condition (H2) was first carried out in \cite{DHI,S1,S2}, where only the global minimum of $\tilde{\Lambda}^{-1}$ was considered. We also refer the reader to \cite{I3}, where the value $\mu_1$ is investigated for a specific elliptic problem.

This article is organized as follows: in section 2 we prove Theorem \ref{THM1}. In section 3 we apply it to two classes of functionals. Finally, section 4 is devoted to applications of our results to several classes of elliptic problems.\\
%Finally, let us note that (H2) is aimed to deal with the case where $\Lambda$ is bounded from above by a negative constant (so that it does not vanish). Whenever $\Lambda$ is bounded from above by a positive constant, one may consider $-\Lambda$ instead of $\Lambda^{-1}$. In this regard, see example ... below.
\subsection*{Acknowledgements} The third author is thankful to Y. Il'yasov, who gave the ideia of applying the Ljusternik-Schnirelman theory to the functional $\Lambda$. \\

\section{Proof of Theorem \ref{THM1}}\strut

Let us collect some properties of $\Lambda$ and $t_0$. Throughout this section we assume that (H1) holds. \begin{lemma}\label{Lambda_0p} \strut
	\begin{itemize}
		\item[i)] $D$ is a symmetric cone and $t_0 \in C^1(D)$ is even. Moreover, $t_0(su)=t_0(u)/s$ for every $u\in D$ and $s>0$.
		\item[ii)] $t_0(u)u \in \mathcal{N}_{\Lambda(u)}^+ \cup \mathcal{N}^-_{\Lambda(u)}$ for every $u\in D$.
		\item[iii)] $\Lambda \in C^1(D)$. Moreover, it is even, $0$-homogeneous and satisfies:
		\begin{equation}\label{pl}
			\Lambda'(u)v=\frac{\Phi_{\Lambda(u)}'(t_0(u)u)(t_0(u)v)}{I_2(t_0(u)u)}, \quad\forall u\in D, \quad\forall v\in X.
		\end{equation}
		In particular, $\Lambda'(u)u=0$ for every $u \in D$, i.e. $D$ is the Nehari set associated to $\Lambda$.
%	\item[iii)] Fix $u\in D$, then for each $w\in X$, there exists a unique $v\in X$ such that $w=t_0(u)v+t'_0(u)vu$.
	\end{itemize}
\end{lemma}
\begin{proof} \strut
	\begin{itemize}
		\item[i)] Let $G:D\times (0,\infty)\to \mathbb{R}$ be given by $G(u,t)=\psi'_u(t)$, so that $G(u,t_0(u))=0$ and $G_t(u,t_0(u))<0$, by (H1). We apply the implicit function theorem to conclude that $t_0 \in C^1(D)$. Now take $u\in D$, $s>0$ and note that $\psi_{su}(t)=\mu_0(t(su))=\psi_{u}(st)$ for any $t>0$, so that $\psi'_{su}(t)=\psi_u'(st)s$, which implies that $t_0(su)=t_0(u)/s$ and $D$ is a cone. In addition,  since $I_1,I_2$ are even, it follows that $\mu_0$ is even, and then $\psi_u(t)=\mu_0(tu)=\mu_0(t(-u))=\psi_{-u}(t)$, for any $t>0$. By the uniqueness condition in (H1) we conclude that $D$ is symmetric, and $t_0(u)=t_0(-u)$. \\
		\item [ii)] Given $u \in D$, from $\psi_u'(t_0(u))=0$ we obtain that
		\begin{eqnarray*}\Phi_{\Lambda(u)}'(t_0(u)u)t_0(u)u&=& t_0(u)\left(I_1'(t_0(u)u)u-\Lambda(u) I_2'(t_0(u)u)u\right)\\&=&t_0(u) I_2(t_0(u)u)\psi_u'(t_0(u))=0.		\end{eqnarray*}
		In addition, one can see that 
		\begin{equation}\label{ej}
		J_{\Lambda(u)}'(t_0(u)u)t_0(u)u=I_2(t_0(u)u)t_0(u)^2\psi_u''(t_0(u))\neq 0.
		\end{equation}\\

		\item[iii)]  First note that since $t_0 \in C^1(D)$ is even it follows that $\Lambda\in C^1(D)$ and it is even. Moreover, $\Lambda(su)=\mu_0(t_0(su)su)=\mu_0(t_0(u)u)=\Lambda(u)$ for any $s>0$ and $u \in D$.
In addition,
		\begin{eqnarray*}
			\Lambda'(u)v&=&\frac{\left(I_2(t_0(u)u)I_1'(t_0(u)u)-I_1(t_0(u)u)I_2'(t_0(u)u)\right)(t_0(u)v+(t'_0(u)v)u)}{I_2(t_0(u)u)^2} \\
			&=&\frac{\left(I_1'(t_0(u)u)-\Lambda(u)I_2'(t_0(u)u)\right)(t_0(u)v+(t'_0(u)v)u)}{I_2(t_0(u)u)}
		\end{eqnarray*}
	for any $u\in D$ and $v \in X$. Since
	\begin{equation*}
		\left(I_1'(t_0(u)u)-\Lambda(u)I_2'(t_0(u)u)\right)(t'_0(u)v)u=t'_0(u)v I_2(t_0(u)u)\psi'_u(t_0(u)) =0,
	\end{equation*}
\eqref{pl} follows from the previous item. Finally, from the previous formula we infer that $\Lambda'(u)u=0$ for any $u \in D$.	
%	\item[iii)] Since $D$ is a cone, it is sufficiently to prove it for the set $\mathcal{N}_0\equiv \{t_0(u)u: u\in D\}=\{u\in D: \mu_0'(u)u=0\}$. We claim that $\mathcal{N}_0$ is a $C^1$ manifold. In fact, note that $(\mu_0'(u)u)'u=\mu_0''(u)vu+\mu_0'(u)v$. If $u\in \mathcal{N}_0$, then by taking $v=u$ and noting that $t_0(u)=1$ we obtain from condition (H2) that $(\mu_0'(u)u)'v=\mu_0''(u)uu+\mu_0'(u)u=\mu_0''(u)uu<0$, therefore $\mathcal{N}_0$ is a $C^1$ manifold. Now, since $(\mu_0'(u)u)'u<0$, it is clear that if $\mathcal{T}_{\mathcal{N}_0}(u)$ denotes the tangent plane of $\mathcal{N}_0$ at $u\in \mathcal{N}_0$, then $X=\mathcal{T}_{\mathcal{N}_0}(u)\oplus \{su: s\in \mathbb{R}\}$. Therefore $w=w_1+su$

	\end{itemize}
\end{proof}

\begin{proposition}\label{propciritcal} \strut
\begin{enumerate}
\item	  If $u\in D$ and $\Lambda'(u)=0$, then $\Phi_{\Lambda(u)}(t_0(u)u)=0$ and $\Phi'_{\Lambda(u)}(t_0(u)u)=0$. 
\item If $u \neq 0$ is  such that $\Phi_\mu'(u)=0$ and $\Phi_\mu(u)=0$, then $\mu=\Lambda(u)$ and $\Lambda'(u)=0$.
\end{enumerate}
\end{proposition}

\begin{proof} \strut
\begin{enumerate}
\item Let $u\in D$ be a critical point of $\Lambda$. The fact that $\Phi_{\Lambda(u)}(t_0(u)u)=0$ follows from the definition of $t_0(u)$ and $\Lambda$. To prove that $\Phi'_{\Lambda(u)}(t_0(u)u)=0$, note that if $w\in X$ then, by taking $v=w/t_0(u)$, we conclude from Lemma \ref{Lambda_0p} that
		\begin{equation*}
		0=I_2(t_0(u)u)\Lambda'(u)v=\Phi'_{\Lambda(u)}(t_0(u)u)w, 
	\end{equation*}
i.e. $\Phi'_{\Lambda(u)}(t_0(u)u)=0$.\\
\item Let $u \neq 0$ be a critical point of $\Phi_\mu$ such that $\Phi_\mu(u)=0$. Then $\mu=\frac{I_1(u)}{I_2(u)}$, and
$\psi_u'(1)=\frac{\Phi_{\mu}'(u)u}{I_2(u)}=0$, i.e. $t_0(u)=1$ and consequently $\Lambda(u)=\mu$. By Lemma \ref{Lambda_0p} we deduce that $u$ is a critical point of $\Lambda$.
\end{enumerate}
	
\end{proof}

\begin{proposition}\label{propciritcalrestricted}  Any critical point of $\tilde{\Lambda}$ is a critical point of $\Lambda$.
		\end{proposition}
		
\begin{proof} Let $u \in S \cap D$ with $\tilde{\Lambda}'(u)=0$. From Lemma \ref{Lambda_0p} we know that $\Lambda'(u)u=0$. Moreover, since $D$ is an open set,  we have  $\tilde{\Lambda}'(u)=\Lambda'(u)_{|\mathcal{T}_{S\cap D}(u)}$, where $\mathcal{T}_{S\cap D}(u)$ is the tangent space to $S\cap D$ at $u$. If $w\in X$, then $w=v+tu$ for some $v\in \mathcal{T}_{S\cap D}(u)$ and $t\in \mathbb{R}$, which implies that $\Lambda'(u)w=\tilde{\Lambda}'(u)v+\Lambda'(u)tu=0$.\\
\end{proof}

We are now ready to prove Theorem \ref{THM1}.\\
\begin{proof}[{\bf Proof of Theorem \ref{THM1}}] \strut	
\begin{enumerate}
\item By Proposition \ref{propciritcal} (ii) we know that if $u \neq 0$ satisfies $\Phi_\mu(u)=\Phi'_\mu(u)=0$ then $\mu=\Lambda(u)$, so that $ \inf_D \Lambda \leq \mu \leq \sup_D \Lambda$, which yields the desired conclusion.\\

\item   First we assume that $\Lambda$ is bounded from below. Since $\gamma(D\cap S)=\infty$ and $\tilde{\Lambda}$ satisfies the Palais-Smale condition at every $\mu_n$,
by the Ljusternick-Schnirelman theorem (see e.g. \cite[Corollary 4.17]{G}) there exists a sequence $(u_{n})\subset S$ such that $\tilde{\Lambda}'(u_n)=0$ and $\tilde{\Lambda}(u_n)=\mu_n$. From the previous propositions, we have that
$$\frac{\partial \Lambda}{\partial u}(c,\pm u_{n,c})=0\quad \text{and}\quad
\Lambda(c,\pm u_{n,c}) = \mu_{n,c} \quad \forall n\in \mathbb{N}.$$
 From  Proposition \ref{propciritcal} and the fact that $u \mapsto t(c,u)$ is even, 
the sequence given by
$$v_{n}:=t(c,u_{n,c})u_{n,c}$$
satisfies
$$\Phi_{\mu_{n,c}}( \pm v_{n,c})= c\quad\text{and}\quad
\Phi'_{\mu_{n,c}}(\pm v_{n,c}) = 0 \quad \forall n\in \mathbb N.$$
Now, if $\Lambda$ is bounded from above, then we deal with the functional $-\Lambda$, which is bounded from below. Moreover,  since $ \tilde{\Lambda}'=-\widetilde{(-\Lambda)}'$  we see that $\widetilde{-\Lambda}=-\tilde{\Lambda}$  satisfies the Palais-Smale condition at the level $\mu$ if and only if $\tilde{\Lambda}$ does so at the level $-\mu$.
Thus  $\displaystyle \inf_{F\in \mathcal{F}_n}\sup_{u\in F} (-\tilde{\Lambda}(u))$ is a critical value of $ -\tilde{\Lambda}$, so that $\displaystyle \sup_{F\in \mathcal{F}_n}\inf_{u\in F}\tilde{\Lambda}(u)=-\displaystyle \inf_{F\in \mathcal{F}_n}\sup_{u\in F} (-\tilde{\Lambda}(u))$ is a critical value of $\tilde{\Lambda}$. 
\end{enumerate}		 
\end{proof}
\section{Some classes of functionals}\strut

Let us apply Theorem \ref{THM1} to two classes of functionals. In the first one $\Lambda$ is bounded from below, whereas for the second one it is bounded from above.

\subsection{A first class of functionals}
Since $\|\cdot \| \in C^1(X \setminus \{0\})$ and $D$ is an open set, we know that $S\cap D=\{u\in D:\ \|u\|=1\}$ is a $C^1$ manifold and $\mathcal{T}_{S\cap D}(u)=\{v \in X: i'(u)v=0\}$, where $i(u)=\frac{1}{2}\|u\|^2$. We consider the functional
\begin{equation*}
\Phi_\mu(u):=\frac{1}{\eta}N(u) -\frac{\mu}{\alpha}A(u)-\frac{1}{\beta}B(u)
\end{equation*}
where  $1<\alpha<\eta<\beta$, and $N,A,B \in C^1(X)$ are even functionals satisfying the following conditions:
\begin{enumerate}
\item $N,A,B$ are  $\eta$-homogeneous, $\alpha$-homogeneous and $\beta$-homogeneous, respectively.
\item There exists $C>0$ such that $ N(u)\geq C^{-1}\|u\|^{\eta}$, $|A(u)| \leq C\|u\|^\alpha$ and $|B(u)|\leq C\|u\|^{\beta}$ for all $u \in X$. 
\item $A'$ and $B'$ are completely continuous, i.e. $A'(u_n) \to A'(u)$ and $B'(u_n) \to B'(u)$ in $X^*$ if $u_n \rightharpoonup u$ in $X$.
\item $A(u)>0$ for any $u \neq 0$ and $\gamma(B^+)=\infty$, where $B^+:=\{u \in X: B(u)>0\}$.
\item $N$ is weakly lower semicontinuous and there exists $C>0$ such that 
$$(N'(u)-N'(v))(u-v)\geq C(\|u\|^{\eta-1}-\|v\|^{\eta-1})(\|u\|-\|v\|)$$ for any $u,v \in X$.\\
\end{enumerate}

We set 
\begin{equation} \label{fn} \mathcal{F}_n=\{F\subset S\cap B^+: F \mbox{ is compact, symmetric, and } \gamma(F)\ge n\}\end{equation} and
\begin{equation*}
	\mu_n:=C_{\alpha,\beta,\eta} \displaystyle
	\inf_{F\in \mathcal{F}_n}\sup_{u\in F}  \frac{N(u)^{\frac{\beta-\alpha}{\beta-\eta}}}{A(u)B(u)^{\frac{\eta-\alpha}{\beta-\eta}}},
\end{equation*}
where  $$C_{\alpha,\beta,\eta}:=\frac{\alpha}{\eta} \frac{\beta-\eta}{\beta-\alpha} \left(\frac{\beta}{\eta}\frac{\eta-\alpha}{\beta-\alpha}\right)^{\frac{\eta-\alpha}{\beta-\alpha}}.$$
In particular, note that $$\mu_1=C_{\alpha,\beta,\eta} \inf_{B^+} \frac{N(u)^{\frac{\beta-\alpha}{\beta-\eta}}}{A(u)B(u)^{\frac{\eta-\alpha}{\beta-\eta}}}.$$

We obtain the following result:

\begin{theorem}\label{CCStructureExistence}
Under the above conditions the following properties hold:
\begin{enumerate}
\item $(\mu_n) \subset (0,\infty)$, $(\mu_n)$ is nondecreasing and $\mu_n \to \infty$ as $n \to \infty$.
\item There is no $u \in X \setminus \{0\}$ such that $\Phi_\mu'(u)=0$ and $\Phi_\mu(u)=0$ for $\mu< \mu_1$. 
\item There exists a sequence $(u_n)\subset B^+$ such that $$\Phi_{\mu_n}(\pm u_n)=0 \quad \mbox{and} \quad \Phi'_{\mu_n}(\pm u_n)=0, \quad \forall n \in \mathbb{N}.$$ Moreover,  $u_n \in \mathcal{N}_{\mu_n}^-$ for every $n \in \mathbb{N}$.
\end{enumerate}
\end{theorem}

\begin{proof} It is clear that $(\mu_n)$ is nondecreasing.  First we apply Theorem \ref{THM1} to prove the second and third items.
For any $u \in X \setminus \{0\}$ we have $$\mu_0(u)=\frac{\alpha}{\eta}\frac{N(u)}{A(u)}-\frac{\alpha}{\beta}\frac{B(u)}{A(u)},$$
so that $$\psi_u(t)=\frac{\alpha}{\eta}\frac{N(u)}{A(u)}t^{\eta-\alpha}-\frac{\alpha}{\beta}\frac{B(u)}{A(u)}t^{\beta-\alpha}.$$
It follows that (H1)  holds with $D=B^+$. Indeed, $\psi_u$ has a critical point $t_0(u)>0$ if and only if $u \in D$. In this case $$t_0(u)=\left(\frac{\beta}{\eta} \frac{\eta-\alpha}{\beta -\alpha}\frac{N(u)}{B(u)}\right)^{\frac{1}{\beta-\eta}},$$
which is a nondegenerate global maximum point of $\psi_u$.
Thus $$\Lambda(u)=C_{\alpha,\beta,\eta} \frac{N(u)^{\frac{\beta-\alpha}{\beta-\eta}}}{A(u)B(u)^{\frac{\eta-\alpha}{\beta-\eta}}}.$$
We see that $\Lambda \in C^1(D)$ with
\begin{equation}\label{ld}
\Lambda'(u)v=C_{\alpha,\beta,\eta}Q(u) \left(\frac{\beta-\alpha}{\beta-\eta}A(u)B(u)N'(u)v-N(u)B(u)A'(u)v-\frac{\eta-\alpha}{\beta-\eta}N(u)A(u)B'(u)v\right),
\end{equation}
where 
\begin{equation}
\label{q} Q(u)=\frac{N(u)^{\frac{\eta-\alpha}{\beta-\eta}} B(u)^{\frac{\alpha-\beta}{\beta-\eta}}}{A(u)^2},
\end{equation}
for any $u\in D$ and $v \in X$. In particular, $\tilde{\Lambda} \in C^1(S\cap D)$.
In addition, it follows from (2) that $\Lambda(u)\geq C$ for some $C>0$ and any $u \in  D$. Thus $\mu_1>0$.

Let us show that $\tilde{\Lambda}$ satisfies the Palais-Smale condition. We pick $(u_n) \subset S\cap D$  such that $(\Lambda(u_n))$ is bounded and $\tilde{\Lambda}'(u_n)\to 0$, i.e. $\left|\Lambda'(u_n)v\right| \leq \epsilon_n \|v\|$ for any $v \in \mathcal{T}_{S\cap D}(u_n)$, with $\epsilon_n \to 0$.
 Since $(u_n)$ is bounded, we have $u_n\rightharpoonup u$ in $X$ up to a subsequence. By (3) we know that $A,B$ are weakly continuous, i.e. $A(u_n) \to A(u)$ and $B(u_n) \to B(u)$. Since $N(u_n)$ is bounded away from zero,  it follows that $u\neq 0$, $A(u)>0$ and $B(u)> 0$, i.e. $u \in D$.  Now observe that
for any $w\in X$ and any $n$, there exist an unique pair $(t_n,v_n)\in \mathbb{R}\times \mathcal{T}_{S\cap D}(u_n)$ such that $w=v_n+t_nu_n$. Hence $i'(u_n)w=t_ni'(u_n)u_n=t_n$, so $(t_n)$ is bounded, and consequently $(v_n)$ is bounded as well. Therefore $\Lambda'(u_n)v_n\to 0$ and since $\Lambda'(u_n)u_n=0$, we conclude that $\Lambda'(u_n)w\to 0$ for any $w \in X$. Taking $w=u_n-u$ in \eqref{ld} and using the fact that $Q(u_n)$ is away from zero, $A'(u_n)(u_n-u)\to 0$, and $B'(u_n)(u_n-u) \to 0$,  we infer that $N'(u_n)(u_n-u) \to 0$. Hence
$(N'(u_n)-N'(u))(u_n-u) \to 0$, and (5) implies that $\|u_n\| \to \|u\|$, so that by the uniform convexity of $X$ we deduce that $u_n\to u$ in $X$. Since $u \in S \cap D$, we obtain the desired conclusion and thus we can apply Theorem \ref{THM1}. 
Lastly, since $I_2(u)=\alpha^{-1}A(u)>0$ for every $u \neq 0$, we infer from \eqref{ej} that $u_n \in \mathcal{N}_{\mu_n}^-$ for every $n$.

Finally,  $\mu_n \ge \mu_1>0$ for every $n$. We employ Lemma \ref{lsi} to show that $\mu_n \to \infty$. Let $(u_n) \subset S$ with $u_n\rightharpoonup 0$ in $X$. Then $N(u_n)$ is away from zero, whereas $A(u_n),B(u_n) \to 0$, so that $\Lambda(u_n) \to \infty$, which yields the conclusion.
 \end{proof}
	
\subsection{A second class of functionals} \label{Sectionsecondclass}
We consider now the functional
\begin{equation*}
\Phi_\mu(u):=\frac{1}{\eta}N(u) +\frac{\mu}{\alpha}A(u)-\frac{1}{\beta}B(u)
\end{equation*}
where 
$1<\eta<\beta<\alpha$, and $N,A,B \in C^1(X)$ are even functionals satisfying:
\begin{enumerate}
	\item $N,A,B$ are  $\eta$-homogeneous, $\alpha$-homogeneous and $\beta$-homogeneous, respectively.
	\item There exists $C>0$ such that $ N(u)\geq C^{-1}\|u\|^{\eta}$ and $|B(u)|\leq C\|u\|^{\beta}$ for all $u \in X$. 
	\item $A'$ and $B'$ are completely continuous, i.e. $A'(u_n) \to A'(u)$ and $B'(u_n) \to B'(u)$ in $X^*$ if $u_n \rightharpoonup u$ in $X$.
	\item $A(u)>0$ for any $u \neq 0$ and $\gamma(B^+)=\infty$, where $B^+:=\{u \in X: B(u)>0\}$.
	%\item $\frac{B(u)^{\frac{\alpha-\eta}{\beta-\eta}}}{A(u)N(u)^{\frac{\alpha-\beta}{\beta-\eta}}}$ is bounded from above in $B^+$. 
	\item $N$ is weakly lower semicontinuous and there exists $C>0$ such that 
	$$(N'(u)-N'(v))(u-v)\geq C(\|u\|^{\eta-1}-\|v\|^{\eta-1})(\|u\|-\|v\|)$$ for any $u,v \in X$.\\
\end{enumerate}
In addition, we shall assume that
\begin{enumerate}
\item [(6)] The set $\{u \in X: \Phi_{\mu}'(u)u=\Phi_\mu(u)=0\}$ is uniformly bounded for $\mu \in [a,b] \subset (0,\infty)$.\\
\end{enumerate}
We set now
\begin{equation*}
	\mu_n:=D_{\alpha,\beta,\eta} \displaystyle
	\sup_{F\in \mathcal{F}_n}\inf_{u\in F}  \frac{B(u)^{\frac{\alpha-\eta}{\beta-\eta}}}{A(u)N(u)^{\frac{\alpha-\beta}{\beta-\eta}}},
\end{equation*}
where  $D_{\alpha,\beta,\eta}:=\frac{\alpha}{\beta} \frac{\beta-\eta}{\alpha-\eta} \left(\frac{\eta}{\beta}\frac{\alpha-\beta}{\alpha-\eta}\right)^{\frac{\alpha-\beta}{\beta-\eta}}$ and $\mathcal{F}_n$ are given by \eqref{fn}. We have now $$\mu_1=D_{\alpha,\beta,\eta} \sup_{u \in B^+}\frac{B(u)^{\frac{\alpha-\eta}{\beta-\eta}}}{A(u)N(u)^{\frac{\alpha-\beta}{\beta-\eta}}}.$$

\begin{theorem}\label{APPCC}
Under the above conditions the following properties hold:
\begin{enumerate}
\item $(\mu_n) \subset (0,\infty)$, $(\mu_n)$ is nonincreasing and $\mu_n \to 0$ as $n \to \infty$.
\item There is no $u \in X \setminus \{0\}$ such that $\Phi_\mu'(u)=0$ and $\Phi_\mu(u)=0$ for $\mu>\mu_1$. 
\item There exists a sequence $(u_n)\subset B^+$ such that $$\Phi_{\mu_n}(\pm u_n)=0 \quad \mbox{and} \quad \Phi'_{\mu_n}(\pm u_n)=0, \quad \forall n \in \mathbb{N}.$$ Moreover,  $u_n \in \mathcal{N}_{\mu_n}^+$ for every $n \in \mathbb{N}$.
\end{enumerate}
\end{theorem}

\begin{proof} First of all, it is clear that $(\mu_n)$ is nonincreasing and $\mu_n>0$ for every $n$.
We have now $$\mu_0(u)=\frac{\alpha}{\beta}\frac{B(u)}{A(u)}-\frac{\alpha}{\eta}\frac{N(u)}{A(u)},$$
so that $$\psi_u(t)=\frac{\alpha}{\beta}\frac{B(u)}{A(u)}t^{\beta-\alpha}-\frac{\alpha}{\eta}\frac{N(u)}{A(u)}t^{\eta-\alpha}$$
has a unique critical point $$t_0(u)=\left(\frac{\beta}{\eta} \frac{\alpha-\eta}{ \alpha-\beta}\frac{N(u)}{B(u)}\right)^{\frac{1}{\beta-\eta}}$$
for any $u \in B^+$.
Thus $$\Lambda(u)=D_{\alpha,\beta,\eta} \frac{B(u)^{\frac{\alpha-\eta}{\beta-\eta}}}{A(u)N(u)^{\frac{\alpha-\beta}{\beta-\eta}}}.$$
We see that $\Lambda \in C^1(B^+)$ with
$$
\Lambda'(u)v=D_{\alpha,\beta,\eta}Q(u) \left(\frac{\alpha-\eta}{\beta-\eta}A(u)N(u)B'(u)v-N(u)B(u)A'(u)v-\frac{\alpha-\beta}{\beta-\eta}B(u)A(u)N'(u)v\right),$$
where $Q$ is given by \eqref{q},
for any $u \in B^+$ and $v \in X$. 

 We claim that $\Lambda$ 
 is bounded from above. Indeed, suppose by contradiction that there exists $\{u_n\}\subset S$ such that $\Lambda(u_n)\to \infty$. From the expression of $\Lambda$ and since $N(u_n)$ is away from zero, we deduce that $A(u_n) \to 0$, i.e. $u_n \rightharpoonup 0$ in $X$, so that $B(u_n) \to 0$. It follows that $t_0(u_n) \to \infty$. On the other hand, since \begin{equation}\label{eto}\Phi_{\Lambda(u_n)}'(t_0(u_n)u_n)t_0(u_n)u_n=\Phi_{\Lambda(u_n)}(t_0(u_n)u_n)=0,\end{equation} and $\Lambda(u_n) \to \infty$, by condition (6) we infer that $(t_0(u_n))$ is bounded, and we reach a contradiction. Thus $\Lambda$ is bounded from above.
 
%Since $(\tilde{\Lambda}^{-1})'=-\frac{\tilde{\Lambda}'}{\tilde{\Lambda}^2}$ and $\tilde{\Lambda}$ is bounded from above, we see that a Palais-Smale sequence for $\tilde{\Lambda}^{-1}$ is a Palais-Smale sequence for $\tilde{\Lambda}$.
Let us show that $\tilde{\Lambda}$ satisfies the Palais-Smale condition at $\mu_n$ for every $n$. Since $\mu_n>0$ for every $n$, it suffices to show the Palais-Smale condition at $\mu>0$.
 If $(u_n) \subset S\cap D$  is  a PS sequence at the level $\mu>0$, then we can assume that $u_n \rightharpoonup u$ in $X$. We use then (7) to prove that $u \in B^+$. Indeed, otherwise by (2) we would have $t_0(u_n)\to \infty$, and we obtain a contradiction as in the previous argument. We argue then as in the proof of Theorem \ref{CCStructureExistence}, to show that $\tilde{\Lambda}$ satisfies the Palais-Smale condition. From Theorem \ref{THM1} the existence part is complete. Lastly, since $I_2(u)=-\alpha^{-1}A(u)<0$ for every $u \neq 0$, we infer from \eqref{ej} that $u_n \in \mathcal{N}_{\mu_n}^+$ for every $n$.

Finally we prove that $\mu_n \to 0$ i.e. $\mu_n^{-1} \to \infty$.  Note that $\mu_n^{-1}=\displaystyle \inf_{F\in \mathcal{F}_n}\sup_{u\in F}(\tilde{\Lambda})^{-1}(u)$. Let $(u_n) \subset S$ with $u_n\rightharpoonup 0$ in $X$, and assume that $(\tilde{\Lambda})^{-1}(u_n)\not \to \infty$, i.e. $\Lambda(u_n) \not \to 0$. From \eqref{eto} and condition (6) we deduce that $(t_0(u_n))$ is bounded.
 On the other hand, the expression of $t_0$ shows that $t_0(u_n)\to \infty$, and we reach a contradiction. Therefore $(\tilde{\Lambda})^{-1}(u_n) \to \infty$ and Lemma \ref{lsi} yields that $\mu_n^{-1} \to \infty$, as desired.
\end{proof}

\begin{remark}\label{r1}\strut
\begin{enumerate}
\item The previous result also holds if instead of assuming that $A'$ is completely continuous, we assume that $A=N^{\sigma}$ for some $\sigma>0$. As a matter of fact, repeating the argument in the proof of Theorem \ref{CCStructureExistence} we still find that $(N'(u_n)-N'(u))(u_n-u) \to 0$. %Moreover, \eqref{ect} still yields a contradiction.
\item Even though $\Lambda$ is bounded from below in the previous proof, we can not exclude that the levels $\displaystyle \inf_{F\in \mathcal{F}_n}\sup_{u\in F}  \Lambda(u)$ are nonzero. Even more, in some cases as the Kirchhoff problem \eqref{sk} below one can show that these levels are indeed equal to zero for every $n$. Moreover, in this case $\inf \Lambda=0$, so that $\Lambda$ does not satisfy the Palais-Smale condition at the level zero.
\end{enumerate}
\end{remark}
\section{Applications}\strut

Next we apply the previous results to several classes of elliptic problems. In the sequel $\mu$ is a real parameter, $\Omega\subset \mathbb{R}^N$ is a bounded domain, and $p^*=\frac{Np}{N-p}$ if $p<N$, with $p^*=\infty$ if $p \ge N$.

\subsection{A concave-convex problem}
Consider the problem
\begin{equation}
	\label{CCequation}
\left\{
\begin{array}
	[c]{lll}%
	-\Delta_p u =\mu g|u|^{q-2}u+f|u|^{r-2}u & \mathrm{in} & \Omega,\\
u=0  & \mathrm{on} & \partial\Omega,
\end{array}
\right. 
\end{equation}
where $\Delta_p$ is the $p$-Laplacian operator, $1<q<p<r<p^*$, and $f,g\in L^\infty(\Omega)$ with $g>0$ in $\Omega$, and $f>0$ in some subdomain $\Omega'\subset \Omega$. Let $\Phi_\mu:W_0^{1,p}(\Omega)\to \mathbb{R}$ be the energy functional associated to \eqref{CCequation}, namely,
\begin{equation*}
	\Phi_\mu(u)=\frac{1}{p}\int_\Omega |\nabla u|^p-\frac{\mu}{q}\int_\Omega g|u|^q-\frac{1}{r}\int_\Omega f|u|^r.
\end{equation*}  
One may easily check that $N(u)=\int_\Omega |\nabla u|^p$, $A(u)=\int_\Omega g|u|^q$, and $B(u)=\int_\Omega f|u|^r$ satisfy the conditions of Theorem \ref{CCStructureExistence} with $\eta=p$, $\alpha=q$, and $\beta=r$.
We have then
\begin{equation*}
	D:=B^+=\left\{u\in W_0^{1,p}(\Omega): \int_{\Omega} f|u|^r>0\right\},
\end{equation*}
so that $W_0^{1,p}(\Omega')\subset D$, and consequently $\gamma(D)=\infty$. Therefore, we infer the following result from Theorem \ref{CCStructureExistence}:

\begin{theorem} Under the previous assumptions there exists a nondecreasing sequence $(\mu_n)\subset(0,\infty)$ with the following properties:
\begin{enumerate}
\item $\mu_n \to \infty$ as $n \to \infty$.
\item  \eqref{CCequation} has no nontrivial weak solution having zero energy for $\mu<\mu_1$.
\item There exists a sequence $(u_n) \subset D$ such that $\Phi_{\mu_n}(\pm u_n)=0$ and $\Phi_{\mu_n}'(\pm u_n)=0$, i.e. $\pm u_n$ are weak solutions of \eqref{CCequation} with $\mu=\mu_n$, having zero energy.
 Moreover $u_n \in \mathcal{N}_{\mu_n}^-$ for every $n$.
 \end{enumerate}
 \end{theorem}

Since the above conditions on $f$ and $g$ imply that \eqref{CCequation} has no positive solution for $\mu>0$ large enough, cf. \cite[Theorem 2.2]{dFGU}, we deduce the following result:

\begin{corollary}
Under the previous conditions there exist infinitely many couples $(\mu_n,u_n)\subset (0,\infty) \times D$ such that $\pm u_n$ are sign-changing weak solutions of \eqref{CCequation} with $\mu=\mu_n$, having zero energy, for every $n$. 
\end{corollary}

\subsection{A Schrödinger-Poisson problem}
Consider the problem
\begin{equation}
	\label{SPequation}
	\left\{
	\begin{array}
		[c]{lll}%
		-\Delta u+\omega u+\mu \phi u=|u|^{p-2}u & \mathrm{in} & \mathbb{R}^3,\\
		-\Delta\phi+a^2\Delta ^2u=4\pi u^2  & \mathrm{in} & \mathbb{R}^3,
	\end{array}
	\right. 
\end{equation}
where $p\in(2,3)$, $\omega>0$, and $a\ge 0$. We look for radial solutions of \eqref{SPequation}, i.e. $u\in H_r^1(\mathbb{R}^3)$ satisfying \eqref{SPequation}. Denote 
\begin{equation*}
	\mathcal{D}_r=\{\phi\in D_r^{1,2}(\mathbb{R}^3):\Delta\phi\in L_r^2(\mathbb{R}^3)\}.
\end{equation*}
It follows that for each $u\in H_r^1(\mathbb{R}^3)$, there exists an unique $\phi_u\in \mathcal{D}_r$ solving the second equation in \eqref{SPequation}, see \cite{SS}. Moreover, if $\Phi_\mu:H_r^1(\mathbb{R}^3)\to \mathbb{R}$ is defined by
\begin{equation*}
	\Phi_\mu(u)=\frac{1}{2}\int_{\mathbb{R}^3} |\nabla u|^2+\frac{\omega}{2}\int_{\mathbb{R}^3} |u|^2+\frac{\mu}{4}\int_{\mathbb{R}^3} \phi_uu^2-\frac{1}{p}\int_{\mathbb{R}^3} |u|^p, 
\end{equation*}
then $\Phi_\mu$ is $C^1$ and its critical points are solutions to \eqref{SPequation}. One may check that 
$N(u)=\int_{\mathbb{R}^3} \left(|\nabla u|^2+\omega |u|^2\right)$, $A(u)=\int_{\mathbb{R}^3} \phi_uu^2$, and $B(u)=\int_{\mathbb{R}^3} |u|^p$ satisfy the conditions of Theorem \ref{APPCC} with $\eta=2$, $\alpha=4$, and $\beta=p$ (for the proof of condition (5) see \cite{SS}). 
Thus we have $D=H_r^1(\mathbb{R}^3) \setminus \{0\}$.

Finally,  let us show for any $0<c<d$ the set $\{u\in H_r^1(\mathbb{R}^3): \Phi_\mu(u)=0, c\leq \mu \leq d\}$ is uniformly bounded.  Indeed, set $D_\mu:=\frac{\mu}{16\pi}-\varepsilon^4>0$ and note that
	\begin{eqnarray*}%\label{M1}
		0=\Phi_\mu(u)&=&\frac{1}{4}\|\nabla u\|_2^2+\frac{1}{4}\|\nabla u\|_2^2+\frac{1}{2}\| u\|_2^2+\frac{\mu}{4}\int_{\mathbb{R}^3} \phi_u u^2-\frac{1}{p}\|u\|_p^p \nonumber\\
		&\ge& \frac{1}{4}\|\nabla u\|_2^2+D_\mu\|\phi_u\|_\mathcal{D}^2+\frac{1}{2}\| u\|_2^2+\frac{\pi\varepsilon^2}{4}\|u\|_3^3-\frac{1}{p}\|u\|_p^p\nonumber\\
		&=&\frac{1}{4}\|u\|^{2}+D_\mu\|\phi_{u}\|_{\mathcal D}^{2}+\int_{\mathbb{R}^3} h(u) \label{M2}
	\end{eqnarray*}
	where
	\begin{equation*}\label{ft}
		h(t)=\frac{1}{4}t^2+\frac{\pi\varepsilon^2}{4}t^3-\frac{1}{p}t^p,\quad \forall\, t>0.
	\end{equation*}
We can choose $\varepsilon>0$ such that $D_\mu>0$ for all $\mu \in [c,d]$. Therefore, the claim follows arguing as in the proof of \cite[Proposition 3.1]{SS}.

From Theorem \ref{APPCC} we infer the following result: 

\begin{theorem} Under the previous assumptions there exists a nonincreasing sequence $(\mu_n)\subset(0,\infty)$ with the following properties:
\begin{enumerate}
\item $\mu_n \to 0$ as $n \to \infty$.
\item  \eqref{SPequation} has no nontrivial radial weak solution having zero energy for $\mu>\mu_1$.
\item There exists a sequence $(u_n) \subset D$ such that $\Phi_{\mu_n}(\pm u_n)=0$ and $\Phi_{\mu_n}'(\pm u_n)=0$, i.e. $\pm u_n$ are weak radial solutions of \eqref{SPequation} with $\mu=\mu_n$, having zero energy.
 Moreover $u_n \in \mathcal{N}_{\mu_n}^+$ for every $n$.
 \end{enumerate}
\end{theorem}

\subsection{A Kirchhoff problem} \label{sk} Consider the problem
\begin{equation}
	\label{KIRequation}
	\left\{
	\begin{array}
		[c]{lll}%
		-(a+\mu \int_{\Omega} |\nabla u|^2) \Delta u =f|u|^{r-2}u & \mathrm{in} & \Omega,\\
		u=0  & \mathrm{on} & \partial\Omega,
	\end{array}
	\right. 
\end{equation}
where $a>0$, $2<r<4$, and $f\in L^\infty(\Omega)$ with $f>0$ in some subdomain $\Omega'\subset \Omega$. Let $\Phi_\mu:H_0^1(\Omega)\to \mathbb{R}$ be the energy functional associated to \eqref{KIRequation}, i.e.
\begin{equation*}\Phi_{\mu}(u)=\frac{a}{2}\int_{\Omega} |\nabla u|^2+\frac{\mu}{4}\left(\int_{\Omega}|\nabla u|^2\right)^2-\frac{1}{r}\int_{\Omega} f|u|^{r},\ u\in H_0^1(\Omega).
\end{equation*}
It is clear that $N(u)=\int_\Omega |\nabla u|^2$, $A(u)=\left(\int_\Omega |\nabla u|^2\right)^2$, and $B(u)=\int_\Omega f|u|^r$ satisfy the conditions of Theorem \ref{APPCC} with $\eta=2$, $\alpha=4$, and $\beta=r$.
Once again we have
\begin{equation*}
	D:=B^+=\left\{u\in W_0^{1,p}(\Omega): \int_\Omega f|u|^r>0\right\}.
\end{equation*}	
We also note that $\mathcal{N}_\mu$ is uniformly bounded for $\mu \in [c,d] \subset (0,\infty)$, since for
$u \in \mathcal{N}_\mu$ we have 
$$a\|u\|^2+\mu \|u\|^4 - \int_\Omega f|u|^r=0.$$
From Theorem \ref{APPCC} and Remark \ref{r1} we infer the following result:

\begin{theorem} Under the previous assumptions there exists a nonincreasing sequence $(\mu_n)\subset(0,\infty)$ with the following properties:
\begin{enumerate}
\item $\mu_n \to 0$.
\item  \eqref{KIRequation} has no nontrivial weak solution having zero energy for $\mu>\mu_1$.
\item There exists a sequence $(u_n) \subset D$ such that $\Phi_{\mu_n}(\pm u_n)=0$ and $\Phi_{\mu_n}'(\pm u_n)=0$, i.e. $\pm u_n$ are weak solutions of \eqref{KIRequation} with $\mu=\mu_n$, having zero energy.
 Moreover $u_n \in \mathcal{N}_{\mu_n}^+$ for every $n$.
 \end{enumerate}

\end{theorem}

For more on Kirchhoff type problems we refer the reader to \cite{PR} and references therein. Let us note that in \cite{S1} (the idea started in \cite{S2})  an abstract framework inspired by the Kirchhoff and Schrödinger-Poisson equations started to be established. We note that the value $\lambda_0^*$ considered in \cite{S1} corresponds to $\mu_1$, and we conjecture that the scenario depicted in Figure \ref{fig:M2} should also hold for \eqref{KIRequation}. In \cite[Theorem 6.4]{S1} it was proved (for $f\equiv 1$) that there exists $\lambda^*>0$ such that for each $\mu\in (0,\lambda^*)$ one can find two positive solutions $u_\mu\in \mathcal{N}_\mu^+$, $v_\mu\in \mathcal{N}_\mu^-$ of \eqref{KIRequation}. When $\lambda=\lambda^*$ we also proved the existence of a positive solution $w_{\lambda^*}\in \mathcal{N}_{\lambda^*}^0$ (recall that $\mathcal{N}^0_\mu:=\{u \in \mathcal{N}_\mu:  J_{\mu}'(u)u = 0\}$, where $J_\mu(u):=\Phi_\mu'(u)u$ for $u \in X$).  These solutions satisfy:
\begin{enumerate}
	\item $\max\{\Phi_{\mu}(u_\mu),0\}<\Phi_\mu(v_\mu)$, for $\mu\in (0,\lambda^*)$.
	\item The maps $\mu \mapsto \Phi_\mu(u_\mu),\Phi_\mu(v_\mu)$ are continuous in $(0,\lambda^*)$.
	\item $\displaystyle\lim_{\mu\to \lambda^*} \Phi_\mu(u_\mu)=\lim_{\mu \to \lambda^*} \Phi_\mu(v_\mu)= \Phi_{\lambda^*}(w_{\lambda^*})$.
	\item For $\mu>\lambda^*$ the Nehari set is empty, and consequently $0$ is the only critical point of $\Phi_\mu$.
\end{enumerate}
Thus the picture is now as in Figure \ref{fig:M3} and we conjecture, at least for the Kirchhoff equation, that the dashed curves in Figure \ref{fig:M2} have a common endpoint.
\begin{figure}[h]
	\centering
	\begin{tikzpicture}[>=latex]
		%x axis
		\draw[->] (-1,0) -- (5,0) node[below] {$\mu$};
		\foreach \x in {}
		\draw[shift={(\x,0)}] (0pt,2pt) -- (0pt,-2pt) node[below] {\footnotesize $\x$};
		%y axis
		\draw[->] (0,-2) -- (0,2) node[left] {$\mbox{Energy}$};
		\foreach \y in {}
		\draw[shift={(0,\y)}] (2pt,0pt) -- (-2pt,0pt) node[left] {\footnotesize $\y$};
		\node[below left] at (0,0) {\footnotesize $0$};
		\draw[red,thick] (0,-2) .. controls (0,-1.5) and (0,0) .. (4,1.5);
		\draw[blue,thick] (0,.5) .. controls (0,0.5) and (0,0.6) .. (3,1.3);
		\draw[blue,thick] (3,1.3) .. controls (3,1.3) and (4,1.5) .. (4,1.5);
		\draw [thick] (1.16,-.1) node[below]{$\lambda_0^*=\mu_1$} -- (1.16,0.1); 
	%	\draw [thick] (3.1,-.1) node[below]{$\lambda_0^*+\varepsilon$} -- (3.1,0.1); 
		\draw [thick] (4,-.1) node[below]{$\lambda^*$} -- (4,0.1); 
		\draw [thick] (-.1,-2) node[left]{$-\infty$} -- (.1,-2); 
		\node[] at (1,1.5) { {\color{blue}$\Phi_\mu(v_\mu)$}};
		\node[] at (3,0.5) { {\color{red}$\Phi_\mu(u_\mu)$}};
	\end{tikzpicture}
	\caption{Positive solution branches depending on $\mu$ for \eqref{KIRequation}.} \label{fig:M3}
	
\end{figure}
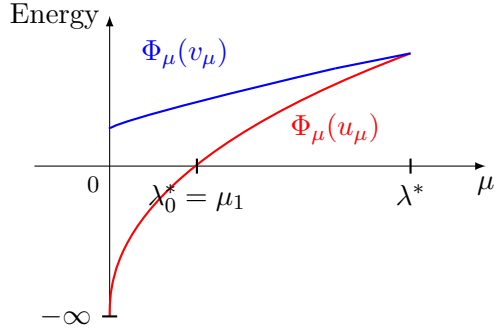
\subsection{A $(p,q)$-Laplacian problem}
 We consider now a situation which is not covered by Theorems \ref{CCStructureExistence} and \ref{APPCC}, namely, the problem
\begin{equation}
	\label{PQequation}
	\left\{
	\begin{array}
		[c]{lll}%
		-\Delta_pu-\Delta_qu =\mu g|u|^{q-2}u+f|u|^{r-2}u & \mathrm{in} & \Omega,\\
		u=0  & \mathrm{on} & \partial\Omega,
	\end{array}
	\right. 
\end{equation}
where  $1<q<p<r<p^*$, and $f,g\in L^\infty(\Omega)$ with $g>0$ in $\Omega$, and $f>0$ in some subdomain $\Omega'\subset \Omega$.  We refer to \cite{MM} for an account on $(p,q)$-Laplacian type problems, i.e. problems involving the operator $-\Delta_p-\Delta_q$.

Let $\Phi_\mu:W_0^{1,p}(\Omega)\to \mathbb{R}$ be the energy functional associated to \eqref{PQequation}, i.e.
\begin{equation*}\Phi_{\mu}(u)=\frac{1}{p}\int_\Omega|\nabla u|^p+\frac{1}{q}\int_\Omega |\nabla u|^q-\frac{\mu}{q}\int_\Omega g|u|^q-\frac{1}{r}\int_\Omega f|u|^{r},\ u\in W_0^{1,p}(\Omega).
\end{equation*}
It is clear that $\Phi_\mu $ is $C^1$. Moreover
\begin{equation*}
	D=\left\{u\in W_0^{1,p}(\Omega): \int_\Omega f|u|^r>0\right\},
\end{equation*}
\begin{equation*}
	t_0(u)=\left(\frac{r}{p}\frac{p-q}{r-q}\frac{\|u\|^p}{\int_\Omega f|u|^{r}}\right)^{\frac{1}{r-p}} \quad \forall u\in D,
\end{equation*}
and
\begin{equation*}
	\Lambda(u)=\frac{q}{p}\frac{r-p}{r-q}\left(\frac{r}{p}\frac{p-q}{r-q}\right)^{\frac{p-q}{r-p}}\frac{\|u\|^p\frac{r-q}{r-p}}{\left(\int_\Omega g|u|^q\right)\left(\int_\Omega f|u|^{r}\right)^{\frac{p-q}{r-p}}}+\frac{\int_\Omega |\nabla u|^q}{\int_\Omega g|u|^q},\quad \forall u\in D.
\end{equation*}

\begin{proposition}\label{H2pq}
$\tilde{\Lambda}$ is bounded from below by a positive constant and satisfies the Palais-Smale condition.
\end{proposition}

\begin{proof}
Note that $\Lambda=\Lambda_{cc}+K$, where $\Lambda_{cc}$ is the corresponding functional associated to the problem \eqref{CCequation} and $K(u):=\frac{1}{q}\frac{\int_\Omega |\nabla u|^q}{\int_\Omega g|u|^q}$ for any $u \in D$.
In particular $\Lambda$ is bounded from below by a positive constant. Let us show that $\tilde{\Lambda}$ satisfies (H2) : we pick $(u_n) \subset S\cap D$  such that $(\Lambda(u_n))$ is bounded and $\tilde{\Lambda}'(u_n)\to 0$, i.e. $\left|\Lambda'(u_n)v\right| \leq \epsilon_n \|v\|$ for any $v \in \mathcal{T}_{S\cap D}(u_n)$, with $\epsilon_n \to 0$.
Since $(u_n)$ is bounded, we have $u_n\rightharpoonup u$ in $W_0^{1,p}(\Omega)$ up to a subsequence. From the expression of $\Lambda$ it is clear that $u \in D$.

 Now observe that for any $w\in  W_0^{1,p}(\Omega)$ and any $n$, there exist an unique pair $(t_n,v_n)\in \mathbb{R}\times \mathcal{T}_{S\cap D}(u_n)$ such that $w=v_n+t_nu_n$. Hence $i'(u_n)w=t_ni'(u_n)u_n=t_n$, so $(t_n)$ is bounded, and consequently $(v_n)$ is bounded as well. Therefore $\Lambda'(u_n)v_n\to 0$ and since $\Lambda'(u_n)u_n=0$, we conclude that $\Lambda'(u_n)w\to 0$. Taking $w=u_n-u$ and noting that 
 \begin{equation*}
 	K'(u)v=\frac{\left(\int_\Omega g|u|^q\right)\int_\Omega |\nabla u|^{q-2}\nabla u \nabla v-\left(\int |\nabla u|^q\right)\int_\Omega g|u|^{q-2}uv}{\left(\int g|u|^q\right)^2}, \quad \mbox{ for } u \in D, v\in W_0^{1,p}(\Omega),
 \end{equation*}
we infer that
\begin{equation*}
	\int_\Omega \left(|\nabla u_n|^{p-2} + |\nabla u_n|^{q-2}\right) \nabla u_n \nabla (u_n-u)\to 0,
\end{equation*}
and hence $u_n\to u$ in $W_0^{1,p}(\Omega)$ with $u \in S \cap D$.
\end{proof}

It is also clear that if $(u_n) \subset S$ and $u_n \rightharpoonup 0$ in $X$ then $\Lambda(u_n) \to \infty$, and $I_2(u)=q^{-1}\int_\Omega g|u|^q>0$ for every $u \neq 0$. 
Thus, Theorem \ref{THM1} yields the following result: 

\begin{theorem} Under the previous assumptions there exists a nondecreasing sequence $(\mu_n)\subset(0,\infty)$ with the following properties:
\begin{enumerate}
\item $\mu_n \to \infty$.
\item  \eqref{PQequation} has no nontrivial weak solution having zero energy for $\mu<\mu_1$.
\item There exists a sequence $(u_n) \subset D$ such that $\Phi_{\mu_n}(\pm u_n)=0$ and $\Phi_{\mu_n}'(\pm u_n)=0$, i.e. $\pm u_n$ are weak solutions of \eqref{PQequation} with $\mu=\mu_n$, having zero energy.
 Moreover $u_n \in \mathcal{N}_{\mu_n}^-$ for every $n$.
 \end{enumerate}
 \end{theorem}

\subsection{A semilinear problem}

Lastly, let us consider the problem
\begin{equation}
	\label{Semiequation}
	\left\{
	\begin{array}
		[c]{lll}%
		-\Delta u=\mu u +|u|^{q-2}u-|u|^{r-2}u & \mathrm{in} & \Omega,\\
		u=0  & \mathrm{on} & \partial\Omega,
	\end{array}
	\right. 
\end{equation}
where $2<q<r<2^*$. The energy functional associated to \eqref{Semiequation} is given by
\begin{equation*}\Phi_{\mu}(u)=\frac{1}{2}\int_\Omega|\nabla u|^2-\frac{\mu}{2}\int_\Omega |u|^2-\frac{1}{q}\int_\Omega |u|^q+\frac{1}{r}\int_\Omega |u|^{r},\quad \forall u\in H_0^1(\Omega).
\end{equation*}
It is clear that $\Phi_\mu $ is $C^1$. Moreover we have that
\begin{equation*}
	D=H_0^1(\Omega)\setminus\{0\},
\end{equation*}
\begin{equation*}
	t_0(u)=\left(\frac{r}{q}\frac{q-2}{r-2}\frac{\|u\|_q^q}{\|u\|_r^r}\right)^{\frac{1}{r-q}}\quad \forall u\in D,
\end{equation*}
and
\begin{equation*}
	\Lambda(u)=\frac{\int_\Omega |\nabla u|^2}{\int_\Omega |u|^2}-\frac{2}{q}\frac{r-q}{r-2}\left(\frac{r}{q}\frac{q-2}{r-2}\right)^{\frac{q-2}{r-q}}\frac{(\|u\|_q^q)^\frac{r-2}{r-q}}{\|u\|_2^2\left(\|u\|_r^r\right)^{\frac{q-2}{r-q}}},\quad \forall u\in D.
\end{equation*}

In the sequel we denote by $\lambda_1(\Omega)$ the first eigenvalue of $(-\Delta,H_0^1(\Omega))$, i.e.
$$\lambda_1(\Omega):=\inf_{u \in H_0^1(\Omega) \setminus \{0\}} \frac{\int_\Omega |\nabla u|^2}{\int_\Omega |u|^2}.$$

\begin{lemma}\label{bounded} $\Lambda$ is bounded from below. Moreover, if $(u_n)\subset S$ and $u_n\rightharpoonup 0$, then $\Lambda(u_n)\to \infty$.\end{lemma}
\begin{proof} Indeed, from the interpolation inequality we have that
	\begin{equation*}
		\|u\|_q\le \|u\|_2^{\frac{2}{q}\frac{r-q}{r-2}}\|u\|_r^{\frac{r}{q}\frac{q-2}{r-2}}, \quad \forall u\in H_0^1(\Omega),
	\end{equation*}
so that
\begin{equation*}
	(\|u\|_q^q)^\frac{r-2}{r-q}\le \|u\|_2^2\left(\|u\|_r^r\right)^{\frac{q-2}{r-q}}, \quad \forall u\in H_0^1(\Omega).
\end{equation*}
%From the inequality $$\frac{\int_\Omega |\nabla u|^2}{\int_\Omega |u|^2} \geq \mu_1(\Omega) \quad \forall u\in H_0^1(\Omega) \setminus \{0\},$$
It follows that 
	\begin{equation}\label{ela}
		\Lambda(u)\ge\frac{\int_\Omega |\nabla u|^2}{\int_\Omega |u|^2}-\frac{2}{q}\frac{r-q}{r-2}\left(\frac{r}{q}\frac{q-2}{r-2}\right)^{\frac{q-2}{r-q}} \quad \forall u\in D,
	\end{equation}
and, as a consequence,	\begin{equation*} \label{com}
		\Lambda(u)\ge C_\Omega:=\lambda_1(\Omega)- \frac{2}{q}\frac{r-q}{r-2}\left(\frac{r}{q}\frac{q-2}{r-2}\right)^{\frac{q-2}{r-q}} \quad \forall u\in D.
	\end{equation*}
	Moreover, it is clear from \eqref{ela} that $\Lambda(u_n)\to \infty$ if $(u_n)\subset S$ and $u_n\rightharpoonup 0$.
\end{proof}

	\begin{proposition}\label{PS} $\tilde{\Lambda}$ satisfies the Palais-Smale condition.
\end{proposition}
\begin{proof} For simplicity we write
		$C=\frac{2}{q}\frac{r-q}{r-2}\left(\frac{r}{q}\frac{q-2}{r-2}\right)^{\frac{q-2}{r-q}}$.
Note that $\Lambda'=K_1-K_2$, where
\begin{equation*}
	K_1(u)v=\frac{\|u\|_2^2\int_\Omega \nabla u\nabla v-\|\nabla u\|_2^2\int_\Omega uv}{\|u\|_2^4}
\end{equation*}
and $K_2$ is given by
\begin{eqnarray*}
	(\|u\|_2^2\left(\|u\|_r^r\right)^{\frac{q-2}{r-q}})^2\frac{K_2(u)v}{C}&=&q\frac{r-2}{r-q}\|u\|_2^2\|u\|_r^{r\frac{q-2}{r-q}}\|u\|_q^{q\frac{q-2}{r-q}}\int_\Omega |u|^{q-2}uv \\ &&-\|u\|_q^{q\frac{r-2}{r-q}}\left(r\frac{q-2}{r-q}\|u\|_2^2\|u\|_r^{r\frac{2q-2-r}{r-q}}\int_\Omega |u|^{r-2}uv+2\|u\|_r^{r\frac{q-2}{r-q}}\int_\Omega uv\right).
\end{eqnarray*}

	Now suppose that $\tilde{\Lambda}(u_n)$ is bounded and $\tilde{\Lambda}'(u_n)\to 0$. Up to a subsequence, we can asusme that $u_n \rightharpoonup u$, and from Lemma \ref{bounded} we have $u\neq 0$. Arguing as in the proof of Theorem \ref{CCStructureExistence} we conclude that $\Lambda'(u_n)\to 0$, so that
	\begin{equation*}
		0=\Lambda'(u_n)(u_n-u)+o_n(1)=-c\Delta u_n(u_n-u)+o_n(1),
	\end{equation*}
for some non-zero constant $c$, which implies that $u_n\to u$ in $S$.
\end{proof}

Since $I_2(u)=\frac{1}{2}\int_\Omega |u|^2>0$ for every $u \neq 0$, we infer the following result:

\begin{theorem}Under the previous conditions there exists a sequence $(\mu_n) \subset (C_\Omega,\infty)$ with the following properties:
\begin{enumerate}
\item $\mu_n \to \infty$.
\item  \eqref{Semiequation} has no nontrivial weak solution having zero energy for $\mu<\mu_1$.
\item There exists a sequence $(u_n) \subset D$ such that $\Phi_{\mu_n}(\pm u_n)=0$ and $\Phi_{\mu_n}'(\pm u_n)=0$, i.e. $\pm u_n$ are weak solutions of \eqref{Semiequation} with $\mu=\mu_n$, having zero energy.
 Moreover $u_n \in \mathcal{N}_{\mu_n}^-$ for every $n$.
 \end{enumerate}
\end{theorem}

\medskip
\appendix
\section{}

\begin{lemma}\label{lsi}
Let $L : S \rightarrow \mathbb{R}$ be bounded from below, and $l_n:=\displaystyle \inf_{F\in \mathcal{F}_n}\sup_{u\in F} L(u)$,
where $\mathcal{F}_n$ is given by \eqref{dfn}, for every $n \in \mathbb{N}$.
Assume in addition that $L(u_n) \to \infty$ if $u_n \rightharpoonup 0$ in $X$. Then  $l_n \to \infty$ as $n \to \infty$.	
\end{lemma}

\begin{proof}
We argue as in \cite{FL} to show that given $T>0$ there exists $n_T\in\mathbb{N}$ such that $\displaystyle \sup_{u\in F}L(u)>T$ for all $n\geq n_L$ and $F\in\mathcal{F}_n$. 
Assume by contradiction that 
	\begin{equation}\label{0contr}
		\exists M>0, k_n \to +\infty \mbox{ and } F_{k_n} \in \mathcal{F}_{k_n} \mbox{ such that  } \sup_{u\in F_{k_n}}L(u)\leq M, \ \forall n\in \mathbb N.
	\end{equation}
	By \cite[Lemma 44.32]{Z}, for every $n$ there exists a finite-dimensional subspace $X_n$ of $X$ and an odd continuous
	operator $P_n: X \to X_n$ such that $\|P_n(u)\|\leq \|u\|$ for all $u\in X$,  and $P_n(u_n)\rightharpoonup u$ in $X$ if $u_n\rightharpoonup u$ in $X$. 
	We claim that there exists $n_0=n(C,M)\in\mathbb{N}$ such that $\|P_{n_0}(u)\|>\delta$ for every  $u\in B_M=\lbrace u\in S: L(u)\leq M\rbrace$ and for some $\delta>0$. Otherwise, there exists a sequence $\{u_n\}\subset B_M$ with $\|P_{n}(u_n)\| \leq {1}/{n}$. Since $\{u_n\}$ is bounded, up to a subsequence, $u_n\rightharpoonup u$ and $P_n(u_n)\rightharpoonup u$ in $X$. On the other hand  $\|P_n(u_n)\|\to 0$ and consequently  $u=0$. By our assumption, we reach the contradiction $L(u_n)\to \infty$.
	 Therefore the claim is proved. Let now $k_0=d_0+1$, with $d_0=\dim X_{n_0}$. By \eqref{0contr} there exists $k_n>k_0$ such that  $F_{k_n}\subset B_M$. Thus $P_{n_0}(u)\neq 0$ for all  $u\in F_{k_n}$. Since $P_n$ is continuous we have that $\gamma(P_{n_0}(F_{k_n}))\geq \gamma(F_{k_n})\geq k_n \geq k_0$. Now, considering the inclusion map $j:P_{n_0}(F_{k_n})\to X_{n_0}$ and an isomorphism $\phi : X_{n_0}\to\mathbb{R}^{d_0}$, we have that $\hat{h}:=\phi\circ j: P_{n_0}(F_{k_n}) \to \mathbb{R}^{d_0}$ is an odd map. Hence $\gamma(P_{n_0}(F_{k_n}))\leq d_0$, and we obtain $d_0\geq \gamma(P_{n_0}(F_{k_n}))\geq k_0=d_0+1$, which is a contradiction. Therefore, \eqref{0contr} cannot occur, and we deduce that $l_n \to+ \infty$ as $n \to \infty$.

\end{proof}

%%%%%%%%%%%%%%%%%%%%%%%%
%\bibliographystyle{apalike}
%\bibliographystyle{plain}
%\bibliography{ref.bib} %%% Nomes dos seus arquivos .bib
%\label{ref-bib}
\end{document}